\documentclass{amsart}
\usepackage{amsmath,amsthm}
\usepackage{amsfonts,amssymb}
\usepackage{enumerate}

\usepackage{srcltx}
\RequirePackage[english]{babel}

\newtheorem{thm}{Theorem}

\newtheorem{cor}{Corollary}
\theoremstyle{remark}

\newtheorem{defn}{Definition}
\newtheorem{exam}{Example}

\newcommand\R{\mathbb{R}}
\newcommand\N{\mathbb{N}}
\newcommand\T{\mathbb{T}}
\newcommand\Z{\mathbb{Z}}

\newcommand\supp{\operatorname{supp}}

\newcommand\Int{\operatorname{int}}

\newcommand\gq{\succeq}
\newcommand\wh{\widehat}

\newcommand\ol{\overline}

\begin{document}

\title[Doubling condition at the origin]{Doubling condition at the origin for
non-negative positive definite functions}

\author{Dmitry~Gorbachev}
\address{D.~Gorbachev, Tula State University,
Department of Applied Mathematics and Computer Science, 300012 Tula, Russia}
\email{dvgmail@mail.ru}
\author{Sergey Tikhonov}
\address{S.~Tikhonov, Centre de Recerca Matem\`{a}tica\\
Campus de Bellaterra, Edifici~C 08193 Bellaterra (Barcelona), Spain; ICREA, Pg.
Llu\'{i}s Companys 23, 08010 Barcelona, Spain, and Universitat Aut\`{o}noma de
Barcelona.}
\email{stikhonov@crm.cat}

\date{\today}

\thanks{The first author was partially supported by the
RFBR (no.~16-01-00308) and the Ministry of Education and Science of the Russian
Federation (no.~5414GZ). The second author was partially supported by MTM
2014-59174-P and 2014 SGR 289.}

\begin{abstract}
We study upper and lower estimates as well as the asymptotic behavior of the
sharp constant $C=C_n(U,V)$ in the doubling-type condition at the origin
\[
\frac{1}{|V|}\int_{V}f(x)\,dx\le C\,\frac{1}{|U|}\int_{U}f(x)\,dx,
\]
where $U,V\subset \R^{n}$ are $0$-symmetric convex bodies and $f$ is a
non-negative positive definite function.
\end{abstract}

\maketitle

\bigskip
\section{Introduction}

Very recently, answering the question posed by Konyagin and Shteinikov related
to a problem from number theory \cite{Sh15}, the first author proved
\cite{Go15} that for any positive definite function $f\colon \Z_{q}\to \R_{+}$
and for any $n\in \Z_{+}$ one has
\[
\sum_{0\le k\le 2n}f(k)\le C\sum_{0\le k\le n}f(k),
\]
where the positive constant $C$ does not depend on $n$, $f$, and $q$. More
precisely, it was proved that $C\le \pi^{2}$.

In this paper we study similar inequalities for a non-negative positive
definite function~$f$ defined on $\R^{n}$, $n\ge 1$, i.e.,
\begin{equation}\label{doubl}
\int_{|x|\le 2R}f(x)\,dx\le C\int_{|x|\le R}f(x)\,dx,\quad R>0,
\end{equation}
for some $C>1$. The latter is the well-known doubling condition at the origin.
The doubling condition plays an important role in harmonic and functional
analysis, see, e.g., \cite{Stein}. Note that very recently
inequality \eqref{doubl} in the one-dimensional
case was studied in \cite{EfGaRe16}.

\begin{defn}
A positive definite function $f\colon \R^{n}\to \R_{+}$ is called double
positive definite function (denoted $f\gq 0$).
\end{defn}

As usual \cite[Chap.~1]{Ru62}, a continuous function $f\in C(\R^{n})$ is
positive definite if for every finite sequence $X\subset \R^{n}$ and every
choice of complex numbers $\{c_{a}\colon a\in X\}$, we have
\[
\sum_{a,b\in X}c_{a}\ol{c_{b}}f(a-b)\ge 0.
\]
By Bochner's theorem \cite[Chap.~1]{Ru62}, $f\in C(\R^{n})$ is positive
definite if and only if there is a non-negative finite Borel measure $\mu$ such
that
\begin{equation}\label{f-mu}
f(x)=\int_{\R^{n}}e(\xi x)\,d\mu(\xi),\quad \xi\in \R^{n},
\end{equation}
where $e(t)=\exp{}(2\pi it)$. For $f\in C(\R^{n})\cap L^{1}(\R^{n})$ it is
equivalent to the fact that the Fourier transform of $f$
\[
\wh{f}(\xi)=\int_{\R^{n}}f(x)e(-\xi x)\,dx
\]
is non-negative. Note also that since any positive definite $f$ satisfies
$f(-x)=\ol{f(x)}$, a double positive definite function is even.

Throughout the paper we assume that $U,V\subset \R^{n}$ be $0$-symmetric closed
convex bodies. For any function $f\gq 0$ we study the inequality
\begin{equation}\label{R-ineq}
\frac{1}{|V|}\int_{V}f(x)\,dx\le C\,\frac{1}{|U|}\int_{U}f(x)\,dx,
\end{equation}
where $|A|$ is the volume of $A$ or the cardinality of $A$ if $A$ is a finite
set. By $C_{n}(U,V)$ we denote the sharp constant in \eqref{R-ineq}, i.e.,
\[
C_{n}(U,V):=\sup_{f\gq 0,\ f\ne 0}
\frac{\frac{1}{|V|}\int_{V}f(x)\,dx}{\frac{1}{|U|}\int_{U}f(x)\,dx}.
\]
 The fact that $C_{n}(U,V)<\infty$ for any
$U$ and $V$ will follow from Theorem~\ref{thm-ub} below.

First, we list the following simple properties of $C_{n}(U,V)$.

\begin{enumerate}
\item A trivial lower bound
\begin{equation}\label{C-1}
C_{n}(U,V)\ge 1,
\end{equation}
since $1\gq 0$;

\item The homogeneity property
\begin{equation}\label{C-hom}
C_{n}(\lambda U,\lambda V)=C_{n}(U,V),\quad \lambda>0,
\end{equation}
since $f_{\lambda}(x)=f(\lambda x)\gq 0$ if and only if $f\gq 0$;

\item The homogeneity estimate
\begin{equation}\label{C-lambda}
C_{n}(U,\lambda V)\ge \lambda^{-n}C_{n}(U,V),\qquad \lambda\ge 1,
\end{equation}
since $V\subset \lambda V$;

\item $C_{n}(U,U)=1$ and if $V\subset U$, then
\[
C_{n}(U,V)\le \frac{|U|}{|V|};
\]

\item The multiplicative estimate
\[
C_{n}(U,V)\le C_{n}(\lambda^{k}U,V)(C_{n}(U,\lambda U))^{k},\quad \lambda\ge
1,\ k\in \Z_{+},
\]
which follows from the chain of inequalities
\begin{align*}
C_{n}(U,V)&\le C_{n}(\lambda U,V)C_{n}(U,\lambda U)\\ &\le
C_{n}(\lambda^{2}U,V)C_{n}(\lambda U,\lambda^{2}U)C_{n}(U,\lambda
U)\\ &=C_{n}(\lambda^{2}U,V)(C_{n}( U,\lambda U))^{2}\le \dots\\
&\le C_{n}(\lambda^{k}U,V)(C_{n}(U,\lambda U))^{k};
\end{align*}
\item A trivial upper bound for the doubling constant:
for fixed $\lambda>1$ and any
$r>\lambda$
\begin{equation}\label{C-log}
C_{n}(U,rU)\le (C_{n}(U,\lambda U))^{\log_{\lambda}r}.
\end{equation}
which follows from the multiplicative estimate.

\end{enumerate}

Bellow we will obtain the upper bound for the constant $C_{n}(U,rU)$, which
depends only on $n$.

We will use the following notation. Let $A+B$ be the Minkowski sum of sets $A$
and $B$, $\lambda A$ be the product of $A$ and the number $\lambda$, and
$B_{R}:=\{x\in \R^{n}\colon |x|\le R\}$ be the Euclidean ball.

\bigskip
\section{The upper estimates}

In what follows, we set
\[
H:=\tfrac{1}{2}U\quad \text{and}\quad K:=V+H.
\]

\begin{thm}\label{thm-ub}
Let $X\subset \R^{n}$ be a finite set of points such that
\begin{equation}\label{V-U-cov}
K\subseteq H+X.
\end{equation}
Then
\[
C_{n}(U,V)\le \frac{|X|\,|U|}{|V|}.
\]
\end{thm}
From the geometric point of view, condition \eqref{V-U-cov} means that the
translates $\{H+a\colon a\in X\}$ of the set $H$ covers the set $K$.

\begin{exam}[\cite{EfGaRe16}]\label{exam-C1}
If $n=1$ and $r\in \N$, then
\[
C_{1}(r):=C_{1}([-1,1],[-r,r])\le 2+\frac{1}{r}.
\]
Indeed, take $H=[-\frac12,\frac12]$, $X=\{-r,-r+1,\dots,r-1,r\}$, and
$K=[-r-\frac12,r+\frac12]=H+X$.
\end{exam}

Let $n\in \N$. There holds (\cite[(6)]{RoZo97})
\begin{equation}\label{N-est}
N(K,H)\le \frac{|K-H|}{|H|}\,\theta(H).
\end{equation}
Here $N(K,H)$ denotes the smallest number of translates of $H$
required to cover $K$ and
\begin{equation}\label{theta-H}
\theta(H)=\inf_{X\subset \R^{n}}\theta(H,X),
\end{equation}
where $\theta(H,X)$ is the covering density of $\R^{n}$ by translates of $H$
\cite[p.16]{Ro64}. In other words, for a discrete set $X$ such that
$\R^{n}\subseteq H+X$ one has $|X\cap A|\,|H|/|A|=\theta(H,X)(1+o(1))$ for a
convex body $A$ such that $|A|\to \infty$.

From \eqref{N-est}, taking into account that $H=-H$, $K-H=V+2H=V+U$, and
$|U|=2^{n}|H|$, we obtain that
\[
N(K,H)\le 2^{n}\,\frac{|V+U|}{|U|}\,\theta(H).
\]
Moreover, it is clear that the best possible result in Theorem~\ref{thm-ub} is
when $X$ is such that $|X|=N(K,H)$. Therefore, we have
\begin{cor}\label{cor-theta}
For $n\ge 1$ and any $U$ and $V$, we have
\[
C_{n}(U,V)\le 2^{n}\,\frac{|V+U|}{|V|}\,\theta(H).
\]
In particular, for $r\ge 1$
\begin{equation}\label{Cn-r}
C_{n}(U, rU)\le 2^{n}(1+r^{-1})^{n}\theta(H).
\end{equation}
\end{cor}

Estimate \eqref{Cn-r} substantially improves \eqref{C-log}. For $n=1$ and $r\ge
1$, we have that $\theta([-\frac12,\frac12])=1$ and $C_{1}(r)\le 2(1+r^{-1})$,
which is similar to the estimate from Example~\ref{exam-C1}.

Note that Rogers \cite{Ro57} proved that
\begin{equation}\label{R-est}
\theta(H)\le n\ln n+n\ln \ln n+5n,\quad n\ge 2.
\end{equation}
Estimate
\eqref{R-est} was slightly improved in \cite{Fe09} as follows
\[
\theta(H)\le n\ln n+n\ln \ln n+n+o(n)\quad \text{as}\quad n\to \infty.
\]

Therefore, we obtain

\begin{cor}\label{cor-ub-as}
We have
\[
C_{n}(U,V)\le 2^{n}(n\ln n+n\ln \ln n+n+o(n))\,\frac{|V+U|}{|V|}\quad
\text{as}\quad n\to \infty.
\]

\end{cor}
In particular, taking $V=rU$, $r\ge 1$, we arrive at the following example.
\begin{exam}\label{exam-Cn}
We have
\begin{equation}\label{zv}
C_{n}(U,rU)\le
2^{n}(n\ln n+n\ln \ln n+n+o(n))(1+r^{-1})^{n}\quad \text{as}\quad n\to \infty.
\end{equation}
\end{exam}

\begin{proof}[Proof of Theorem~\ref{thm-ub}]
Consider the function
\[
\varphi:=\varphi_{H}=|H|^{-1}\cdot \chi_{H}*\chi_{H},
\]
where $\chi_{H}$ is the characteristic function of $H$ and
$(f*g)(x)=\int_{\R^{n}}f(x-y)g(y)\,dy$ is the convolution of $f$ and $g$.

Since $\varphi\gq 0$, $\supp \varphi\subset U$, and $\varphi\le \varphi(0)=1$,
we have for any $f\gq 0$
\[
I:=\int_{\R^{n}}f(x)\varphi(x)\,dx=\int_{U}f(x)\varphi(x)\,dx\le \int_{U}f(x)\,dx.
\]

Let $X\subset \R^{n}$ be a finite set and
\[
S(x)=\frac{1}{|X|}\sum_{a\in X}\varphi(x-a).
\]
Then $S\ge 0$ and $\wh{S}=\wh{\varphi}D$, where
\[
D(\xi)=\frac{1}{|X|}\sum_{a\in X}e(a\xi)
\]
is the Dirichlet kernel with respect to $X$.

Let us estimate the integral $I$ from below. Using $f(x)=f(-x)$, we get
\[
\int_{V}f(x)S(x)\,dx\le \int_{\R^{n}}f(x)S(x)\,dx=
\int_{\R^{n}}f(x)S_{0}(x)\,dx:=I_{1},
\]
where $S_{0}(x)=2^{-1}(S(x)+S(-x))$. Taking into account that
\[
\wh{S_{0}}(\xi)=\wh{\varphi}(\xi)\,\frac{D(\xi)+D(-\xi)}{2}=
\wh{\varphi}(\xi)\,\frac{1}{|X|}\sum_{a\in X}\cos{}(2\pi a\xi)\le
\wh{\varphi}(\xi),\quad \xi\in \R^{n},
\]
and using \eqref{f-mu}, we obtain
\[
I_{1}=\int_{\R^{n}}\wh{S_{0}}(\xi)\,d\mu(\xi)\le
\int_{\R^{n}}\wh{\varphi}(\xi)\,d\mu(\xi)=\int_{\R^{n}}f(x)\varphi(x)\,dx=I,
\]
provided that $f$ and $\varphi$ are even.

Let $K=V+H\subseteq H+X$. This means that for any points $x\in V$ and $y\in H$
there is $a\in X$ such that $x+y\in H+a$. Hence,
\[
\sum_{a\in X}\chi_{H}(x+y-a)\ge 1.
\]
Using $H=-H$, we have
\[
\varphi(x)=\frac{1}{|H|}\int_{H}\chi_{H}(x+y)\,dy.
\]
Therefore, for any $x\in V$
\begin{align*}
S(x)&=\frac{1}{|X|}\sum_{a\in X}\frac{1}{|H|}\int_{H}\chi_{H}(x-a+y)\,dy\\
&\ge \frac{1}{|X||H|}\int_{H}\sum_{a\in X}\chi_{H}(x-a+y)\,dy\\
&\ge \frac{1}{|X||H|}\int_{H}dy=\frac{1}{|X|}.
\end{align*}
Thus, combining the estimates above, we arrive at the inequality
\[
\frac{1}{|X|}\int_{V}f(x)\,dx\le \int_{V}f(x)S(x)\,dx\le I\le\int_{U}f(x)\,dx,
\]
which is the desired result.
\end{proof}

\bigskip
\section{The lower estimates}

Our goal is to improve the trivial lower estimate \eqref{C-1}. The idea is to
consider the functions $\sum_{a,b\in X\cap B_{R}}\delta(x+a-b)$, where $X$ is a
packing of $\R^{n}$ by $H$ and $R\gg 1$ (see also \cite{GoTi16, EfGaRe16}).

First we consider the one-dimensional result, partially given in
Example~\ref{exam-C1}.

\begin{thm}[\cite{EfGaRe16}]\label{thm-lb-1}
For $r\in \N$, we have
\[
2-\frac{1}{r}\le C_{1}(r)\le 2+\frac{1}{r},
\]
and $\lim_{r\to \infty}C_{1}(r)=2$.
\end{thm}
This is one of the main results of the paper \cite{EfGaRe16}. The upper bound
is given in Example~\ref{exam-C1}. The lower bound follows from Theorem
\ref{thm-lb-n} below for $U=[-1,1]$, $V=[-r,r]$, and $\Lambda=\Z$. The fact
that $\lim_{r\to \infty}C_{1}(r)=2$ follows from estimates of $C_{1}(r)$ for
integers $r$ and~\eqref{C-lambda}.

Now we consider the general case $n\ge 1$. Our aim is to improve the trivial
lower bound \eqref{C-1} respect to $n$.

Let
\[
\delta_{L}(H)=\sup_{\Lambda\subset \R^{n}}\delta(H,\Lambda),
\]
where $\delta(H,\Lambda)$ is the packing density of $\R^{n}$ by lattice
translates of $H$ \cite[Intr.]{Ro64}. In other words, $\Lambda=M\Z^{n}\subset
\R^{n}$ is a lattice of rank $n$ ($M\in \R^{n\times n}$ is a~generator matrix
of~$\Lambda$, $\det M\ne 0$) such that $a-b\notin \Int{}(2H)$ for any $a,b\in
\Lambda$, $a\ne b$, and $|\Lambda\cap A|\,|H|/|A|=\delta(H,\Lambda)(1+o(1))$
for a convex body $A$ such that $|A|\to \infty$. Note that in this case
$H+\Lambda$ is \textit{a lattice packing} of $H$ \cite[Sect.~30.1]{Gr07}. Recall that
$H=\tfrac{1}{2}U$.

\begin{thm}\label{thm-lb-n}
Let $H+\Lambda$ be a lattice packing of $H$. Then
\begin{equation}\label{Cn-L}
C_{n}(U,V)\ge \frac{|\Lambda\cap \Int V|\,|U|}{|V|}.
\end{equation}

In particular,
\begin{equation}\label{Cn-delta}
C_{n}(U,V)\ge 2^{n}\delta_{L}(H)(1+o(1))\quad \text{as}\quad |V|\to \infty.
\end{equation}
\end{thm}

\begin{proof}[Proof of Theorem~\ref{thm-lb-n}]
Let $\Lambda$ be an lattice with the packing density $\delta(H,\Lambda)$.
Denote $\Lambda_{N}=\Lambda\cap B_{N}$ for $N>0$. Let $B_{r}$ be the smallest
ball that contained $V$. Assume that $R\ge r$ is sufficiently large number and
$\varepsilon$ is sufficiently small. Define
$\varphi_{\varepsilon}=\varphi_{B_{\varepsilon}}$.

We consider the function
\[
f(x)=\sum_{a,b\in \Lambda_{R}}\varphi_{\varepsilon}(x+a-b).
\]
It is easy to see that
\[
f(x)=\sum_{c\in \Lambda_{2R}}N_{c}\varphi_{\varepsilon}(x+c),
\]
where
\[
N_{c}=\sum_{a-b=c}1=\sum_{a\in \Lambda_{R}\cap (\Lambda_{R}+c)}1=|\Lambda_{R}\cap
(\Lambda_{R}+c)|.
\]

Since $\Lambda$ is a lattice, we have $\Lambda=\Lambda+c$ for any $c\in
\Lambda$. Hence, $N_{0}=|\Lambda_{R}|$ and $N_{c}\ge |\Lambda_{R-r}|$ for
$|c|\le r$, provided $\Lambda_{R-r}\subset \Lambda_{R}\cap (\Lambda_{R}+c)$.

On the one hand, since $2H=U$ and $c\notin \Int U$ if $c\in \Lambda\setminus
\{0\}$, we have
\[
\int_{(1-\varepsilon)U}f(x)\,dx=N_{0}=|\Lambda_{R}|.
\]
On the other hand, since $V\subset B_{r}$, we obtain
\[
\int_{(1+\varepsilon)V}f(x)\,dx\ge \sum_{c\in \Lambda_{2R}\cap V}N_{c}\ge
|\Lambda_{R-r}|\,|\Lambda\cap V|.
\]
Therefore,
\[
C_{n}\bigl((1-\varepsilon)U,(1+\varepsilon)V\bigr)\ge
\frac{(1-\varepsilon)^{n}}{(1+\varepsilon)^{n}}\,
\frac{|\Lambda_{R-r}|}{|\Lambda_{R}|}\,\frac{|\Lambda\cap V|\,|U|}{|V|}.
\]
Replacing $V$ by $\tfrac{1-\varepsilon}{1+\varepsilon}V$ and using
\eqref{C-hom} and \eqref{C-lambda} as above, we arrive at
\[
C_{n}(U,V)\ge
\frac{|\Lambda_{R-r}|}{|\Lambda_{R}|}\,\frac{|\Lambda\cap
\tfrac{1-\varepsilon}{1+\varepsilon}V|\,|U|}{|V|}.
\]
Letting $R\to \infty$ and $\varepsilon\to 0$ concludes the proof of
\eqref{Cn-L}.

Inequality \eqref{Cn-delta} follows easily from \eqref{Cn-L} and the definition
of $\delta_{L}(H)$.
\end{proof}

\begin{exam}\label{exam-B}
We consider the balls $U=B_{1}$ and $V=B_{r}$, $r>1$. It
is known that
\[
\delta_{L}(B_{1})\ge c_{n}2^{-n},
\]
where $c_{n}\ge 1$ is the Minkowski constant. It was recently proved in
\cite{Ve13} that $c_{n}>65963n$ for every sufficiently large $n$ and there
exist infinitely many dimensions~$n$ for which $c_{n}\ge 0.5n\ln \ln n$.
\end{exam}

\begin{cor}\label{cor-lb}Let $n\in\N$. We have
\begin{equation}\label{zvzv}
C_{n}(B_{1},B_{r})\ge c_{n}(1+o(1))\quad \text{as}\quad r\to \infty.
\end{equation}
\end{cor}

Comparing \eqref{zv} and \eqref{zvzv} for fixed $n$ and $r\to \infty$, one
observes the exponential gap between the upper and lower estimates of
$C_{n}(B_{1},B_{r})$ with respect to $n$. Let us give examples of $U$ for which the
upper and lower estimates of $C_{n}(U,V)$ coincide.

\begin{exam}\label{exam-L}
Let $H$ be a convex body and $\Lambda$ be a lattice. The set $H+\Lambda$ is
lattice tiling if it is both a packing and a covering \cite[Sect.~32]{Gr07}. In
this case $H$ is \textit{a~tile} and $\delta_{L}(H)=\theta_{L}(H)=1$, where
$\theta_{L}(H)$ is the lattice covering density, cf.~\eqref{theta-H}.
To define $\theta_{L}(H)$, we take the infimum in \eqref{theta-H} over all lattices $\Lambda\subset
\R^{n}$ of rank $n$. Note that $\theta(H)\le \theta_{L}(H)$.

For example, the Voronoi polytop
\[
V(\Lambda)=\{x\in \R^{n}\colon |x|\le |x-a|,\ \forall\,a\in \Lambda\}
\]
of a lattice $\Lambda$ is a tile. In particular, $V(\Z^{n})$ is the cube
$[-\frac12,\frac12]^{n}$.

From Corollary \ref{cor-theta} and Theorem \ref{thm-lb-n}, we have

\begin{thm}\label{thm-tile}
Let $n\in \N$ and $U$ be a tile. We have
\[
C_{n}(U,V)=2^{n}(1+o(1))\quad \text{as}\quad |V|\to \infty.
\]
\end{thm}
\end{exam}

\bigskip
\section{Final remarks}

\smallbreak
\textbf{1.} The inequality
\[
\frac{1}{|V|}\int_{V}f(x)\,dx\le C_{n}(U,V)\,\frac{1}{|U|}\int_{U}f(x)\,dx
\]
holds for any $1$-periodic function $f\gq 0$.
In this case we assume that $U,V\subseteq \T^{n}$, where $\T=\R/\Z.$

Since a positive definite $f$ is such that $f(-x)=\ol{f(x)}$, then
$|f|^{p}\gq 0$ for any $p=2k, k\in \N$. Hence, we obtain the following
$L^{p}$-analogue:
\[
\frac{1}{|V|}\int_{V}|f(x)|^{p}\,dx\le
C_{n}(U,V)\,\frac{1}{|U|}\int_{U}|f(x)|^{p}\,dx.
\]
For $U\subset V=\T^{n}$, this inequality is the well-known Wiener estimate for
positive definite periodic functions (see \cite{Sh75,Hl81,GoTi16}):
\begin{equation}\label{sharpp}
\int_{\T^{n}}|f(x)|^{p}\,dx\le
W_{n,p}(U)\,\frac{1}{|U|}\int_{U}|f(x)|^{p}\,dx,
\end{equation}
which is valid only for $p=2k, k\in \N$. Here, $W_{n,p}(U)$ is a sharp constant
in \eqref{sharpp}. It is clear that
\[
W_{n,2k}(U)\le C_{n}(U,\T^{n}).
\]
It is interesting to compare the known upper bounds of $W_{n,2k}(U)$ and $C_{n}(U,\T^{n})$.
In \cite{GoTi16} it was shown that
\[
W_{n,2k}(rB_{1})\le 2^{(0.401\ldots +o(1))n},\quad r\in (0,1/2).
\]
On the other hand, by Corollary \ref{cor-ub-as}, we obtain that
\[
C_{n}(rB_{1},\T^{n})\le 2^{n(1+o(1))}(1+2r)^{n}.
\]
The exponential gap in the last two bounds is related to the restriction to the
class of functions under consideration.

\smallbreak
\textbf{2.} If $f\gq 0$, then $f^{p}\gq 0$ for any $p\in \N$. This gives
\[
\frac{1}{|V|}\int_{V}(f(x))^{p}\,dx\le
C_{n}(U,V)\,\frac{1}{|U|}\int_{U}(f(x))^{p}\,dx,\quad p\in \N.
\]
It would be of interest to investigate this inequality for any positive $p$; see in
this direction the paper \cite{fit}.

\smallbreak
\textbf{3.} As we showed above, any function $\,f\gq 0$ satisfies the doubling
property at the origin~\eqref{doubl}. However, taking any nontrivial function
$f\gq 0$ such that $f|_{A}=0$, where $A$ is a ball, we can see that the
doubling property may fail outside the origin.

\end{document}